\documentclass[a4paper, reqno, 14pt]{amsart}

\usepackage[usenames,dvipsnames]{color}
\usepackage{amsthm,amsfonts,amssymb,amsmath,amsxtra, appendix}
\usepackage[all]{xy}
\SelectTips{cm}{}
\usepackage{xr-hyper}
\usepackage{verbatim}
\usepackage{mathrsfs}
\usepackage{diagbox}
\usepackage{enumerate}
\usepackage[shortlabels]{enumitem}
\usepackage{tikz}
\usepackage{tkz-euclide}
\usepackage{setspace} 

\newcommand{\remind}[1]{{\bf ** #1 **}}

\RequirePackage{xspace}
\RequirePackage{etoolbox}
\RequirePackage{varwidth}
\RequirePackage{enumitem}
\RequirePackage{tensor}
\RequirePackage{mathtools}
\RequirePackage{longtable}
\RequirePackage{multirow}

\usepackage{tikz}
\usepackage{tikz-cd}
\usepackage{csquotes}

\usepackage[url=false,doi=false,isbn=false,sorting=nyt,giveninits=true,maxbibnames=4]{biblatex}
\def\ge{\geqslant}
\def\le{\leqslant}
\def\a{\alpha}
\def\b{\beta}
\def\g{\gamma}
\def\G{\Gamma}
\def\d{\delta}

\def\e{\epsilon}

\def\th{\theta}

\def\l{\lambda}

\def \ud{\underline}
\def\pr{\mathrm{pr}}
\def\af{\mathrm{af}}

\def\<{\langle}
\def\>{\rangle}

\newcommand{\bG}{\mathbf G}

\newcommand{\BF}{\ensuremath{\mathbb {F}}\xspace}
\newcommand{{\BG}}{\ensuremath{\mathbb {G}}\xspace}

\newcommand{{\BK}}{\ensuremath{\mathbb {K}}\xspace}

\newcommand{\BQ}{\ensuremath{\mathbb {Q}}\xspace}
\newcommand{\BR}{\ensuremath{\mathbb {R}}\xspace}
\newcommand{\BS}{\ensuremath{\mathbb {S}}\xspace}

\newcommand{\BZ}{\ensuremath{\mathbb {Z}}\xspace}

\newcommand{\CA}{\ensuremath{\mathcal {A}}\xspace}

\newcommand{\CC}{\ensuremath{\mathcal {C}}\xspace}

\newcommand{\CI}{\ensuremath{\mathcal {I}}\xspace}

\newcommand{\CK}{\ensuremath{\mathcal {K}}\xspace}

\newcommand{\CS}{\ensuremath{\mathcal {S}}\xspace}

\newcommand{\CW}{\ensuremath{\mathcal {W}}\xspace}

\DeclareMathOperator{\Adm}{Adm}

\newcommand{\wt}{\mathrm{wt}}

\def\tW{\tilde W}

\DeclareMathOperator{\supp}{supp}

\newtheorem{theorem}{Theorem}
\newtheorem{proposition}[theorem]{Proposition}
\newtheorem{lemma}[theorem]{Lemma}

\newtheorem{corollary}[theorem]{Corollary}

\theoremstyle{definition}
\newtheorem{definition}[theorem]{Definition}
\newtheorem{example}[theorem]{Example}
\newtheorem*{example*}{Example}
\newtheorem{thmintro}{Theorem}
\newtheorem{propintro}[thmintro]{Proposition}

\newtheorem{remark}[theorem]{Remark}

\newtheorem*{function*}{Function}

\numberwithin{equation}{section}
\numberwithin{theorem}{section}

\setitemize[0]{leftmargin=*,itemsep=\the\smallskipamount}
\setenumerate[0]{leftmargin=*,itemsep=\the\smallskipamount}

\renewcommand{\to}{%
   \ifbool{@display}{\longrightarrow}{\rightarrow}%
   }
\let\shortmapsto\mapsto
\renewcommand{\mapsto}{%
   \ifbool{@display}{\longmapsto}{\shortmapsto}%
   }
\newlength{\olen}
\newlength{\ulen}
\newlength{\xlen}
\newcommand{\xra}[2][]{%
   \ifbool{@display}%
      {\settowidth{\olen}{$\overset{#2}{\longrightarrow}$}%
       \settowidth{\ulen}{$\underset{#1}{\longrightarrow}$}%
       \settowidth{\xlen}{$\xrightarrow[#1]{#2}$}%
       \ifdimgreater{\olen}{\xlen}%
          {\underset{#1}{\overset{#2}{\longrightarrow}}}%
          {\ifdimgreater{\ulen}{\xlen}%
             {\underset{#1}{\overset{#2}{\longrightarrow}}}
             {\xrightarrow[#1]{#2}}}}%
      {\xrightarrow[#1]{#2}}
   }
\makeatother
\newcommand{\xyra}[2][]{%
   \settowidth{\xlen}{$\xrightarrow[#1]{#2}$}%
   \ifbool{@display}%
      {\settowidth{\olen}{$\overset{#2}{\longrightarrow}$}%
       \settowidth{\ulen}{$\underset{#1}{\longrightarrow}$}%
       \ifdimgreater{\olen}{\xlen}%
          {\mathrel{\xymatrix@M=.12ex@C=3.2ex{\ar[r]^-{#2}_-{#1} &}}}%
          {\ifdimgreater{\ulen}{\xlen}%
             {\mathrel{\xymatrix@M=.12ex@C=3.2ex{\ar[r]^-{#2}_-{#1} &}}}
             {\mathrel{\xymatrix@M=.12ex@C=\the\xlen{\ar[r]^-{#2}_-{#1} &}}}}}%
      {\mathrel{\xymatrix@M=.12ex@C=\the\xlen{\ar[r]^-{#2}_-{#1} &}}}%
   }
\makeatletter
\newcommand{\xla}[2][]{%
   \ifbool{@display}%
      {\settowidth{\olen}{$\overset{#2}{\longleftarrow}$}%
       \settowidth{\ulen}{$\underset{#1}{\longleftarrow}$}%
       \settowidth{\xlen}{$\xleftarrow[#1]{#2}$}%
       \ifdimgreater{\olen}{\xlen}%
          {\underset{#1}{\overset{#2}{\longleftarrow}}}%
          {\ifdimgreater{\ulen}{\xlen}%
             {\underset{#1}{\overset{#2}{\longleftarrow}}}
             {\xleftarrow[#1]{#2}}}}%
      {\xleftarrow[#1]{#2}}
   }
\newcommand{\isoarrow}{%
   \ifbool{@display}{\overset{\sim}{\longrightarrow}}{\xrightarrow\sim}%
   }

\begin{document}

\title[]{Irreducibility of Local Models}

\author[Xuhua He]{Xuhua He}
\address{Department of Mathematics and New Cornerstone Science Laboratory, The University of Hong Kong, Pokfulam, Hong Kong, Hong Kong SAR, China}
\email{xuhuahe@hku.hk}

\author[Qingchao Yu]{Qingchao Yu}
\address{Institute for Advanced Study, Shenzhen University, Nanshan District, Shenzhen, Guangdong, China}
\email{qingchao\_yu@outlook.com}

\thanks{}



\begin{abstract}
In this paper, we consider the geometric special fibers of local models of Shimura varieties and of moduli of $\bG$-Shtukas with parahoric level structure. We investigate two problems with respect to the irreducibility of local models. First, we classify the cases where the local models are irreducible. Next, we show that the fibers of the level-changing map between the geometric special fiber of local models with different parahoric levels are always isomorphic to single (i.e., irreducible) Schubert varieties in the partial flag variety.
\end{abstract}

\maketitle


\section*{Introduction}

\subsection{Background}
Local models serve as projective flat schemes that capture the singularities of integral models of Shimura varieties and moduli of $\bG$-Shtukas with parahoric level structures. First introduced by Rapoport and Zink \cite{RZ96} for PEL-type Shimura varieties, these models were defined using moduli spaces of self-dual lattice chains in specific skew-Hermitian vector spaces. More general constructions for Shimura varieties and moduli of $\bG$-Shtukas, based on group-theoretic approaches, were developed by Zhu \cite{Zhu}, Pappas and Zhu \cite{PZ13} for tamely ramified groups, and Fakhruddin, Haines, Lourenço, and Richarz \cite{FHLR} for wildly ramified groups.

The generic fiber of the local model is a single affine Schubert variety in the affine Grassmannian. The special fiber is more complicated and is a certain union of the affine Schubert varieties in the partial affine flag varieties. It is defined as follows. 

Let $k = \BF_q$ be a finite field, and let $L = \overline{k}((t))$. Let $\mathbf{G}$ be a connected reductive group over $L$, and let $\breve{G} = \mathbf{G}(L)$. Let $\mu$ be a dominant cocharacter of $\mathbf{G}$. Let $K$ be a spherical subset of affine simple reflections, and let $\breve{\CK}$ be the corresponding standard parahoric subgroup of $\breve{G}$. Define
\begin{align*}
\CA_{K}(\mathbf{G}, \mu) = \breve{\CK} \Adm(\mu) \breve{\CK} / \breve{\CK} \subset \breve{G}/\breve{\CK}.
\end{align*}
Here $\Adm(\mu)$ denotes the \emph{$\mu$-admissible subset} of the Iwahori-Weyl group of $\breve{G}$. See Section \ref{sec:1.2} for the precise definition. Note that $\CA_{K}(\mathbf{G}, \mu)$ is a closed finite-dimensional subvariety of the partial affine flag variety $\breve{G}/\breve{\CK}$. The geometric special fibers of local models are known to be isomorphic to some of these varieties $\CA_{K}(\mathbf{G}, \mu)$.

\subsection{Main Results} In this paper, we establish the following two results on the irreducibility of $\CA_{K}(\bG,\mu)$.

We first present a result on the irreducible components of $\CA_{K}(\bG,\mu)$.

\begin{thmintro}[Theorem \ref{thm:irr}] There is a natural bijection between the irreducible components of $\CA_{K}(\bG,\mu)$ and the double quotient $W_{\pr(K)}\backslash W_0/W_{I(\mu)}$ of the relative Weyl group.

Moreover, if $\mathbf{G}$ is quasi-simple over $L$ and $\mu$ is non-central, then $\CA_{K}(\bG,\mu)$ is irreducible if and only if $W_0 = W_{\pr(K)} W_{\text{short}}$ and $W_{\text{short}}$ stabilizes $\mu$. 
\end{thmintro}

We refer to \S\ref{sec:1.2} for the notation used here. In Proposition \ref{prop:class}, we also obtain an explicit classification of the cases where $\CA_{K}(\bG,\mu)$ is irreducible. In particular, if $\CA_{K}(\bG,\mu)$ is irreducible, then $\breve{\CK}$ is a maximal parahoric subgroup, and if moreover $\breve{K}$ is not special, then the Dynkin diagram is not of simply-laced type.

Our second main result concerns the fibers of the level-changing maps of the local models. Let $K_1\subseteq K_2$ be two spherical subsets. The natural projection $\breve G/\breve \CK_1\longrightarrow \breve G / \breve \CK_2$ induces the level-changing map
$\pi =\pi_{K_1,K_2} : \CA_{K_1}(\bG,\mu) \longrightarrow \CA_{K_2}(\bG,\mu)$.

\begin{thmintro}[Theorem \ref{thm:main}]\label{thm:01}
For any spherical subsets $K_1 \subseteq K_2$, the fibers of the map $\pi_{K_1,K_2}: \CA_{K_1}(\bG,\mu) \longrightarrow \CA_{K_2}(\bG,\mu)$ are single (i.e. irreducible) Schubert varieties in the partial flag variety.
\end{thmintro}
By Theorem \ref{thm:irr}, neither $\CA_{K_1}(G, \mu)$ nor $\CA_{K_2}(G, \mu)$ is irreducible in general. It is surprising that the fibers of the level-changing map are always irreducible.

Note that the admissible set $\Adm(\mu)$ is a complicated combinatorial object. To prove Theorem \ref{thm:01}, we use several results on the admissible set established in \cite{He16} and the ``acute cone criterion" given by Haines and the first-named author \cite{HH17}. We reduce Theorem \ref{thm:01} to the following combinatorial result on the admissible set.

\begin{propintro}[Proposition \ref{prop:max}]\label{prop:02}
Let $K$ be a spherical subset of $\tilde \BS$. Let $w\in \Adm(\mu) $. Then the set $wW_K \cap \Adm(\mu)$ contains a unique maximal element.    
\end{propintro}

Explicitly describing the element $\max(w W_K \cap \Adm(\mu))$ in Proposition \ref{prop:02} remains a challenging problem. In Theorem \ref{thm:algorithm}, we provide an efficient algorithm to compute this maximal element when $\breve K$ is maximal special. The primary tools used include the quantum Bruhat graph and the Demazure product.

\smallskip

\noindent \textbf{Acknowledgments:} The questions on the irreducibility of local models arose during the collaboration of the first-named author with Ulrich G\"ortz and Michael Rapoport. We thank Felix Schremmer for pointing out the algorithm in \S\ref{sec:3.3} in the Sagemath source code and Mark Shimozono for providing us with the reference \cite{BM15}. We also thank Felix Schremmer for many helpful discussions and comments. XH is partially supported by the New Cornerstone Science Foundation through the New Cornerstone Investigator Program and the Xplorer Prize, as well as by the Hong Kong RGC grant 14300122.

\section{Irreducible components}\label{sec:1}

\subsection{Group theoretic data.}\label{sec:1.1}
Let $k$ be an algebraically closed field and $L=k((t))$ be the field of formal Laurent series over $k$. Let $G$ be a connected reductive group over $L$. Let $\G_0$ be the absolute Galois group of $L$.

Let $S$ be a maximal $L$-split torus of $G$ defined over $L$. Let $T$ be the centralizer of $S$ in $G$. By Steinberg’s theorem, $G$ is quasi-split over $L$. Then $T$ is a maximal torus. We denote by $N_T$ the normalizer of $T$. Let $W_0=N_T(L)/T(L)$ be the (relative) Weyl group. Let $X_*(T)$ be the cocharacter group of $T$ and let $V=X_*(T)_{\G_0} \otimes \BR = X_*(S)\otimes\BR$. 

Let $\CA$ be the apartment of $G$ corresponding to $S$. We fix an alcove $\mathfrak{a}$ of $\CA$ and call it the base alcove. Let $\breve \CI$ be the Iwahori subgroup corresponding to $\mathfrak{a}$. Let $\tW= N_T(L)/(T(L) \cap  \breve \CI)$ be the Iwahori--Weyl group. Choosing a special vertex of $\mathfrak{a}$, by \cite[Proposition 13]{HR}, we can identify $\CA$ with $V$, which gives rise to a natural isomorphism
$$\tW \cong X_*(T)_{\G_0} \rtimes W_0=\{t^{\ud\l} w \mid \ud\l \in X_*(T)_{\G_0}, w \in W_0\}.$$
Let $\tilde \BS$ be the index set of affine simple reflections in $\tW$ determined by $\mathfrak{a}$. In other words, the set of affine simple reflections in $\tW$ is $\{s_i\mid i\in \tilde \BS\}$. Let $\BS_0\subset\tilde \BS$ be the index set of simple reflections in $W_0$. Let $W_{\af}$ be the affine Weyl group, i.e. the subgroup of $\tW$ generated by all affine simple reflections. Let $\ell$ be the corresponding length function and $\le$ be the Bruhat order of $\tW$ with respect to the base alcove $\mathfrak{a}$. For any $w\in \tW$, we fix a representative $\dot{w}\in N_T(L)$. 

For any subset $K$ of $\tilde \BS$, we denote by $W_K$ the subgroup of $\tilde W$ generated by simple reflections in $K$. We say that $K$ is {\it spherical} if $W_K$ is finite. We denote by ${}^K \tW$ (resp. $\tW^K$) the set of minimal representatives of $W_K \backslash \tW$ (resp. $\tW/W_K$) in $\tW$. For any $J, K \subset \tilde \BS$, we simply write ${}^J \tW^K$ for ${}^J \tW \cap \tW^K$. 

Let $\Phi$ be the relative root system of $G$ over $L$. The dominant chamber, which we denote by $C^+$, is by convention the unique Weyl chamber containing the base alcove $\mathfrak{a}$. Let $X_*(T)^+\subseteq X_*(T)$ be the set of dominant cocharacters. 

Let $\Sigma$ be the reduced root system associated with $\Phi$ as in \cite[\S1.7]{Tits1979} and let $\Sigma^+\subseteq \Sigma$ be the set of positive roots determined by the dominant chamber $C^+$. Then $W_{\af}$ equals the affine Weyl group of the root system $\Sigma$. Note that $\Sigma$ might not agree with $\Phi$ even if $\Phi$ is reduced. We have the natural pairing $\<-,-\>$ between $\BR\Sigma$ and $V$.
\subsection{The set $\CA_{K}(\bG,\mu)$}\label{sec:1.2}

Let $\mu\in X_*(T)^+$ be a dominant cocharacter of $\mathbf G$ and $\underline \mu$ its image in $X_*(T)_{\G_0}$. The $\mu$-admissible set is defined as
$$\Adm(\mu) = \{w\in\tW\mid w\le t^{x(\ud\mu)}\text{ for some }x\in W_0\}.$$ 

Let $K$ be a spherical subset of $\tilde\BS$ and let $\breve \CK$ be the standard parahoric subgroup of $\breve G$ corresponding to $K$. Let $\Adm(\mu)_K = W_K\Adm(\mu)W_K$. Define
\begin{align*}
\CA_{K}(\bG,\mu) = \bigcup_{w\in \Adm(\mu)} \breve \CK \dot{w} \breve \CK / \breve \CK = \bigcup_{w\in \Adm(\mu)_{K}} \breve \CI \dot{w} \breve \CI / \breve \CK.
\end{align*}
This is a finite-dimensional closed subvariety of the partial affine flag variety $\breve G/\breve \CK$.

The goal of this paper is to study certain geometric properties of $\CA_{K}(\bG,\mu)$ (or equivalently, the geometric special fiber of the local model). In Theorem \ref{thm:irr}, we give a specific description of the set of irreducible components of $\CA_{K}(\bG,\mu)$. In Theorem \ref{thm:main}, we study the fibers between the natural projection map between $\CA_{K}(\bG,\mu)$ with different level structures.

\subsection{Translation elements in $\Adm(\mu)$}\label{sec:1.3}
The following result was first discovered in \cite[Theorem 6.1]{He16}, motivated by the Kottwitz-Rapoport conjecture on the non-emptiness of affine Deligne-Lusztig varieties. A different proof was obtained later in the joint work of Haines and the first-named author in \cite[Proposition 5.1]{HH17}.
\begin{theorem}\label{thm:AdmmuK}
We have
$$\Adm(\mu)_K\cap \tW^{K} = \Adm(\mu)\cap \tW^{K}.$$
\end{theorem}

The following proposition was due to the first-named author and Nie.
\begin{proposition}[{\parencite[Proposition 2.1]{HN17}}]\label{prop:maxK}
The set of maximal elements of $\Adm(\mu)\cap \tW^{K}$ is $\{t^{x(\ud\mu)}\mid x\in W_0, t^{x(\ud\mu)} \in \tW^{K}\}$.
\end{proposition}

We now give a more explicit interpretation of the condition $t^{x(\ud\mu)}\in \tW^{K}$ in Proposition \ref{prop:maxK}. We first introduce the set of affine roots.

Let $\Phi_{\mathrm{af}}\subseteq \Phi\times\BQ$ be the set of affine roots constructed as in \cite[\S1.6]{Tits1979}. We have an identification $\Phi_{\af}\cong \Sigma\times \BZ$ as in \cite{HR}. For any $\g\in \Sigma$, set
\begin{align*}
\d_\g = 
    \begin{cases}
        0,&\g\in\Sigma^+;\\
        1,&\g\in\Sigma^-.
    \end{cases}
\end{align*}
By convention, we choose the set of positive affine roots as $\Phi_{\af}^+ = \{(\a,k)\mid \a\in\Sigma, k\ge \d_{\a} \}$. For any $w = t^{\ud{\l}}y \in   \tW$, the action of $w$ on an affine root $\tilde{\a} = (\a,k)\in \Sigma\times\BZ \cong \Phi_{\af}$ is given by $w(\a,k) = (y(\a), k-\<\ud{\l},y(\a)\>)$. For any affine root $\tilde{\a} = (\a,k)$, the corresponding root hyperplane is $H_{\tilde\a} = \{v\in V \mid \<v,\a\> =-k\}$. \footnote{In literature such as \cite{GHKR10} and \cite{He14}, the positive affine roots are $\{(\a,k)\mid \a\in\Sigma, k\ge \d_{-\a} \}$, and the formula $k -\<\ud{\l},x(\a)\>$ is replaced by $k +\<\ud{\l},x(\a)\>$ in the definition of the action.} 

For any $i\in\BS_0$, let $\a_i\in \Sigma^+$ be the corresponding simple root. For any $i\in \tilde \BS - \BS_0$, set $\a_i = -\th_i$, where $\th_i$ is the highest root in the irreducible component of $\Sigma$ containing $i$. For any $i\in\tilde\BS$, define $\tilde{\a}_i = (\a_i,\d_{\a_i}) \in \Phi_{\af}^+$, the simple affine root.

Note that for any $w\in\tW$ and $i\in \tilde \BS$, $ws_i>w \text{ if and only if } w\tilde{\a}_i\in\Phi_{\af}^+$. Hence, for any $\ud\l \in X_*(T)_{\G_0}$, we have
\begin{align*}\tag{1.1}\label{eq:1.1}
    t^{\ud{\l}}\in \tW^{K} \iff& t^{\ud \l} (\a_i,\d_{\a_i}) \in \Phi_{\af}^+ \text{ for any } i\in K\\
                           \iff& (\a_i, \d_{\a_i} - \<\ud \l,\a_i\>) \in \Phi_{\af}^+ \text{ for any } i\in K\\
                           \iff& \d_{\a_i} - \<\ud \l,\a_i\> \ge \d_{\a_i} \text{ for any } i\in K\\
                           \iff& \<\ud \l,\a_i\> \le 0 \text{ for any } i\in K.
\end{align*}
Let $$C_K = \{v\in V\mid \<v,\a_i\> > 0 \text{ for any }i\in K\}.$$ 
Then (\ref{eq:1.1}) implies that
\begin{align*}\tag{1.2}\label{eq:1.2}
    t^{\ud{\l}}\in \tW^{K} \iff  \ud\l \in -\overline{C_K}.
\end{align*}
Here, $\overline{C_K}$ is the closure of $C_K$. Let $W_{\pr(K)}$ be the image of $W_K$ under the natural projection map $\pr: \tW \longrightarrow W_0$. Observe that $-C_K$ is a Weyl chamber of $W_{\pr(K)}$. Hence, for any $\ud\l \in X_*(T)_{\G_0}$, there exists a unique $W_{\pr(K)}$-conjugate of $\ud\l$ that lies in $-\overline{C_K}$. Combined with (\ref{eq:1.2}), we obtain the following lemma.
\begin{lemma}\label{lem:WK-conjugate}
Let $\ud\l \in X_*(T)_{\G_0}$. Then there exists a unique $W_K$-conjugate (or equivalently, $W_{\pr(K)}$-conjugate) of $t^{\ud\l}$ that lies in $\tW^{K}$.
\end{lemma}

\subsection{Main result}
In this subsection, we give a parameterization of $\mathrm{Irr}\CA_{K}(\bG,\mu) $, the set of irreducible components of the variety $\CA_{K}(\bG,\mu)$.

We say that a spherical subset $K' \subseteq \tilde{\BS}$ is special if the complement of $K' $ in $\tilde\BS$ is a set of special vertices, one from each irreducible component of $\Sigma$. Note that $K' $ is special if and only if $W_{\pr(K' )} = W_0$ (see \cite[\S1.9]{Tits1979}). Otherwise, we say that $K'$ is non-special.

Let $W_{\text{short}}$ be the subgroup of $W_0$ generated by the simple reflections associated with the short roots of $\Sigma$ (if $\Sigma$ is simply-laced, then $W_{\mathrm{short}}$ is the trivial group). The following table describes the type of $W_{\mathrm{short}}$ case by case. 
\begin{center}
\begin{tabular}{ |c|c|} 
 \hline
 Type of $W_0$ & Type of $W_{\text{short}}$ \\ 
 \hline
 Simply-laced type & Trivial group \\
 \hline
 $B_n$ & $A_1$  \\
 \hline 
 $C_n$&  $A_{n-1}$ \\
 \hline
 $F_4$&  $A_2$ \\
 \hline  
 $G_2$& $A_1$ \\
 \hline
\end{tabular}
\end{center}

Let $K$ be a spherical subset of $\tilde \BS$ and $W_{\pr(K)}$ be the image of $W_K$ under the natural projection map $\pr: \tW \longrightarrow W_0$. Let $I(\mu) = \{i\in\BS_0\mid \<\ud\mu,\a_i\> = 0\}$. 

We state the main result of this section. 
\begin{theorem}\label{thm:irr}
We have a natural bijection \begin{align*}
 \Xi:   W_{\pr(K)} \backslash W_0/W_{I(\mu)} \tilde{\longrightarrow} \mathrm{Irr}\CA_{K}(\bG,\mu), \quad [x]\longmapsto  \overline{\breve \CK t^{x(\ud \mu)} \breve \CK/\breve \CK}.
\end{align*}
Moreover, if $\mathbf G$ is quasi-simple over $L$ and $\mu$ is non-central, then $\CA_{K}(\bG,\mu)$ is irreducible if and only if 
\begin{enumerate} 
\item $K$ is maximal special; 
or 
\item $W_0=W_{\pr(K)} W_{\text{short}}$ and $W_{\text{short}}$ stabilizes $\mu$.
\end{enumerate}
\end{theorem}

We have the following more explicit classification. 

\begin{proposition}\label{prop:class}
    Assume that $\mathbf G$ is quasi-simple over $L$ and $\mu$ is non-central. Then $W_0=W_{\pr(K)} W_{\text{short}}$ if and only if $K$ is maximal special or is in the following table: \begin{center}\label{table}
\begin{tabular}{ |c|c|} 
 \hline
 Type & $K$\\ 
 \hline
 $ B_n$ & $\tilde\BS-\{ n\}$ \\
 \hline 
 $C_n$&  $\tilde\BS -\{i\}$ \text{ for some } $i \ne 0,n$ \\
 \hline
 $F_4$&   $\tilde\BS - \{4\}$ \\
 \hline  
 $G_2$& $\tilde\BS - \{2\}$ \\
 \hline
\end{tabular}
\end{center}
\end{proposition}

\begin{remark}
\begin{enumerate}
    \item The type $C_n$ case in the above table contains the Siegel case, which is well-known among experts.
    \item We observe that in the above table, $K$ always equals $\tilde\BS - \{\text{a short root}\}$. However, in the case of type $F_4$, $K = \tilde\BS - \{3\}$ is not in the table. Indeed, in this case, the element $w = s_3s_2s_1$ is not in $W_{\pr(K)} W_{\text{short}}$.
\end{enumerate}

\end{remark}

\subsection{Description of the irreducible components}In this section, we prove that the map $\Xi:[x]\longmapsto  \overline{\breve \CK t^{x(\ud \mu)} \breve \CK/\breve \CK}$ in Theorem \ref{thm:irr} is bijective. For $x\in W_0$, we denote by $[x]$ the image of $x$ in $W_{\pr(K)} \backslash W_0 / W_{I(\mu)}$. For $w\in\tW$, we denote by $\<w\>$ the image of $w$ in $W_K \backslash \tW / W_K$.

First, we prove that $\Xi$ is independent of the choice of the representative. Let $x, x'\in W_0$ be such that $x' = u x z$ for some $u\in W_{\pr(K)}$ and $z \in  W_{I(\mu)}$. Note that $u = \pr(w)$ for some $w\in W_K$. Then $t^{x'(\ud \mu)} = t^{u x z(\ud \mu)} = t^{u x(\ud \mu)} = w t^{x(\ud\mu)} w^{-1}$. Thus, we have $\breve \CK t^{x'(\ud \mu)} \breve \CK = \breve \CK t^{x(\ud \mu)} \breve \CK$.

Second, we prove that for any $x\in W_0$, $\overline{\breve\CK t^{x(\ud\mu)} \breve \CK}$ is an irreducible component of $\CA_K(\bG,\mu)$. Note that 
$$\CA_K(\bG,\mu) = \bigsqcup_{\<w\> \in W_K \backslash \Adm(\mu)_K / W_K } \breve\CK \dot{w} \breve\CK /\breve \CK.$$
Furthermore, the closure relation is given by the induced Bruhat order on $W_K \backslash \tW / W_K$. Therefore, we have
$$\mathrm{Irr}\CA_K(\bG,\mu) = \left\{ \overline{\breve\CK \dot{w} \breve\CK /\breve \CK} \mid \<w\> \text{ maximal in }W_K \backslash \Adm(\mu)_K / W_K\right\}.$$
It suffices to prove that for any $x\in W_0$, $\<t^{x(\ud\mu)}\>$ is maximal in $W_K \backslash \Adm(\mu)_K / W_K$. By Lemma \ref{lem:WK-conjugate}, there exists a (unique) $W_K$-conjugate of $t^{x(\ud\mu)}$ that lies in $\tW^K$. Without loss of generality, we may assume that $t^{x(\ud\mu)}\in\tW^K$. Suppose $t^{x(\ud\mu)}\le w'$ for some maximal element $w'$ in $\Adm(\mu)_K$. By Theorem \ref{thm:AdmmuK}, we can write $w' = w_1 w_2$ with $w_1 \in \Adm(\mu) \cap \tW^{K}$ and $w_2 \in W_K$. Since $t^{x(\ud\mu)}\in\tW^K$, $t^{x(\ud\mu)} \le w_1$. By Proposition \ref{prop:maxK}, we must have $w_1 = t^{x(\ud\mu)}$. Thus, we have $\<t^{x(\ud\mu)}\> = \<w_1\>=\<w_1w_2\> = \<w'\>$. This proves that $\overline{\breve\CK t^{x(\ud\mu)} \breve \CK}$ is an irreducible component and hence $\Xi$ is well-defined.

By definition of $\Adm(\mu)$, a maximal element in $W_K\backslash \Adm(\mu)_K/W_K$ must be of the form $\<t^{x(\ud\mu)}\>$ for some $x\in W_0$. The surjectivity of $\Xi$ then follows.


Finally, we prove that $\Xi$ is injective. Assume that $\overline{\breve \CK t^{x(\ud \mu)} \breve \CK/\breve \CK} = \overline{\breve \CK t^{x'(\ud \mu)} \breve \CK/\breve \CK}$ for some $x,x'\in W_0$. Then $\<t^{x(\ud \mu)}\> = \<t^{x'(\ud \mu)}\>$ in $W_K \backslash \tW / W_K$. By Lemma \ref{lem:WK-conjugate} again, we may assume that $t^{x'(\ud \mu)}$ lies in $\tW^K$. Then $t^{x'(\mu)} = u t^{x(\mu)} v$ for some $u,v \in W_K$. Hence $t^{x'(\ud\mu)}v^{-1}u^{-1} = t^{\pr(u)x(\ud\mu)}$. Since $t^{x'(\ud\mu)}\in\tW^{K}$ and $\ell(t^{x'(\ud\mu)}) = \ell(t^{\pr(u)x (\ud\mu)})$, we have $v^{-1}u^{-1} = 1$ and $x'(\ud\mu) = \pr(u)x(\ud\mu)$. Hence $[x'] = [x]$ in $W_{\pr(K)}\backslash W_0 / W_{I(\mu)}$. This proves the injectivity of $\Xi$.

The proof of the ``moreover" part of Theorem \ref{thm:irr} will be given in the next two subsections.

\subsection{Irreducibility}
We have shown that $\CA_K(\bG,\mu)$ is irreducible if and only if $W_{\pr(K)} \backslash W_0 / W_{I(\mu)}$ is a singleton, or equivalently, $W_0 = W_{\pr(K)} W_{I(\mu)}$. In particular, the property of whether $\CA_K(\bG,\mu)$ is irreducible depends only on the triple $(W_0,I(\mu),K)$. Note that if $K$ is special, then $W_{\pr(K)} = W_0$ and the condition holds automatically. 

By \cite[Corollary 3.4 (ii)]{Dyer93}, each coset of $W_{\pr(K)}$ in $W_0$ has a unique minimal element. Let ${}^{\pr(K)}W_0$ be the set of minimal representatives of the right cosets of $W_{\pr(K)}$. In particular, we have $W \cong W_{\pr(K)} \times  {}^{\pr(K)}W_0$. Set $\supp\left({}^{\pr(K)}W_0 \right) = \bigcup_{x\in {}^{\pr(K)}W_0 }\supp(x)$. Here, $\supp(x)$ is the set of $i\in \BS_0$ such that $s_i$ appears in some (or any) reduced expression of $x$.

\begin{lemma}\label{lem:supp}
Let $K$ be a spherical subset of $\tilde\BS$ and $J$ be a proper subset of $\BS_0$. Then $W_0 =  W_{\pr(K)}   W_{J}$ if and only if $ {}^{\pr(K)}W_0 \subseteq W_{J}$, or equivalently, $\supp\left({}^{\pr(K)}W_0 \right) \subseteq J$.
\end{lemma}

\begin{proof}
The ``if" direction follows from the fact that $W_0 =W_{\pr(K)} {}^{\pr(K)}W_0 $. We now prove the ``only if" direction. Suppose $W_0 = W_{\pr(K)}   W_J$. Let $x \in {}^{\pr(K)}W_0$. We shall prove that $x\in W_J$. 

We introduce more notation. Let $\Sigma_{\pr(K)}$ be the root subsystem of $\Sigma$ spanned by $\{\a_i \mid i \in K\}$ (see \S\ref{sec:1.3} for the definition of $\a_i$) and let $\Sigma_{\pr(K)}^+ = \Sigma_{\pr(K)} \cap \Sigma^+$. Then ${}^{\pr(K)}W_0 = \{x\in W_0 \mid x^{-1}(\b)\in \Sigma^+ \text{ for any }\b\in \Sigma_{\pr(K)}^+\}$. Define $\Sigma_J^+$ similarly. Let $L = \Sigma_{\pr(K)}^+ \cap \Sigma_{J}^+$ and ${}^{L} (W_{J}) = \{z\in W_{J} \mid z^{-1}(\b)\in \Sigma^+ \text{ for any }\b\in L\}$. 

Since $W_0 =  W_{\pr(K)} W_{J}$, we can write $x = u z$ with $u \in W_{\pr(K)}$ and $z \in W_{J}$. Next, we can write $z = z_1 z_2$ with $z_1 \in W_L$ and $z_2 \in {}^{L} (W_{J})$. We claim $z_2 \in {}^{\pr(K)}W_0$. In fact, suppose $z_2^{-1}(\b) \in \Sigma^-$ for some $\b \in \Sigma_{\pr(K)}^+$. Since $z_2 \in W_{J}$, $\b$ must lie in $L$. It contradicts the fact that $z_2 \in {}^{L} (W_J)$. This proves the claim. Note that $x = u z_1 z_2$, $x, z_2 \in {}^{\pr(K)}W_0$ and $u, z_1 \in W_{\pr(K)}$. We must have $x = z_2 \in W_{J}$. This completes the proof.
\end{proof}

\subsection{Classification}
In this subsection, we classify all pairs $(W_0, K)$ where 

(*) $W_0$ is irreducible and $K$ is a non-special spherical subset such that $\supp\left({}^{\pr(K)}W_0 \right)$ is a proper subset of $\BS_0$.

We shall prove that $(W_0,K)$ satisfies (*) if and only if it is one of those listed in the following table.

\begin{center}\label{table2}
\begin{tabular}{ |c|c|c|} 
 \hline
 Type of $W_0$ & $K$ & $\supp\left({}^{\pr(K)}W_0 \right)$\\ 
 \hline
 $ B_n$ & $\tilde\BS-\{ n\}$  &    $   \{n\}$  \\
 \hline 
 $C_n$&  $\tilde\BS -\{i\}$ for some $i \ne 0,n$ &  $\{1,2,\ldots,n-1\}$\\
 \hline
 $F_4$&   $\tilde\BS - \{4\}$ &     $\{3,4\}$   \\
 \hline  
 $G_2$& $\tilde\BS - \{2\}$ &  $\{2\}$  \\
 \hline
\end{tabular}
\end{center}

Observe that in the above table, the set $\supp\left({}^{\pr(K)}W_0 \right)$ always equals the set of short roots. Hence, by Lemma \ref{lem:supp}, we conclude that if $W_0$ is irreducible and $\mu$ is noncentral, then $W_0 =  W_{\pr(K)}  W_{I(\mu)}$ if and only if (i) $K$ is special or (ii) $\supp\left({}^{\pr(K)}W_0 \right) = \text{set of short roots} \subseteq I(\mu) $. This proves the ``moreover" part of Theorem \ref{thm:irr} and Proposition \ref{prop:class}.

We now begin to prove the above classification result. Our strategy is as follows. 

Recall that $C^+$ is the dominant Weyl chamber of $W_0$ and $C_K = \{v\in V\mid \<v,\a_i\> > 0 \text{ for any }i\in K\}$. Note that $C_K$ is a Weyl chamber of the Weyl group $W_{\pr(K)}$. Let $C_K^+$ be the unique Weyl chamber of $W_{\pr(K)}$ that contains the dominant Weyl chamber $C^+$ of $W_0$. Then it is easy to see that 
\begin{align*}\tag{1.3}\label{eq:1.3}
    {}^{\pr(K)}W_0 = \{x\in W_0 \mid  x(C^+)\subseteq C_K^+ \}.
\end{align*}
In the picture below, we draw the regions $C_K$ and $C_K^{+}$ of type $C_2$ for certain $K$.

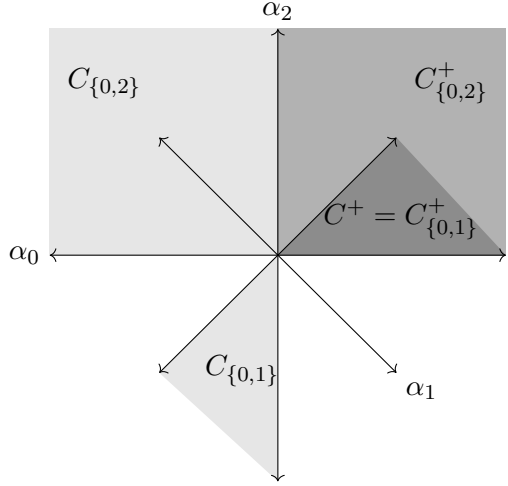
\begin{center}
\begin{tikzpicture}[scale=0.6]
    \coordinate (O) at (0,0);    
\coordinate (A) at (5,0);  
\coordinate (B) at (2.6,2.6);  
\coordinate (C) at (0,5);  
\coordinate (D) at (-2.6,2.6);  
\coordinate (E) at (-5,0);
\coordinate (F) at (-2.6,-2.6);
\coordinate (G) at (0,-5);
\coordinate (H) at (2.6,-2.6);
\coordinate (I) at (-5,5);
\coordinate (J) at (5,5);
\begin{scope}  
    \clip (C) -- (I) -- (E) -- (O) -- cycle;  
    \fill[gray!20] (O) circle (20cm); 
\end{scope}  
\begin{scope}  
    \clip (F) -- (G) -- (O) -- cycle;  
    \fill[gray!20] (O) circle (20cm); 
\end{scope}  
\begin{scope}  
    \clip (A) -- (J) -- (C) -- (O) -- cycle;  
    \fill[gray!60] (O) circle (20cm); 
\end{scope}  
\begin{scope}  
    \clip (A) -- (B) -- (O) -- cycle;  
    \fill[gray!90] (O) circle (20cm); 
\end{scope}  
\draw[->] (O) -- (A) node[above] {};  
\draw[->] (O) -- (B) node[right] {};  
\draw[->] (O) -- (C) node[above] {\large $\alpha_2$};  
\draw[->] (O) -- (D) node[right] {};  
\draw[->] (O) -- (E) node[left] {\large $\alpha_0$};  
\draw[->] (O) -- (F) node[right] {};  
\draw[->] (O) -- (G) node[right] {};  
\draw[->] (O) -- (H) node[below right] {\large $\alpha_1$}; 
\node[above right, font=\large] at (-1.8,-3.1) {$C_{\{0,1\}}$}; 
\node[ font=\large] at (-3.8,3.8) {$C_{\{0,2\}}$}; 
\node[ font=\large] at (3.8,3.8) {$C_{\{0,2\}}^+$}; 
\node[  font=\large] at (2.7,0.8) { $C^+ = C_{\{0,1\}}^+$}; 
\end{tikzpicture}  
\end{center}

We have the following two observations. 
\begin{enumerate}[(a)]
    \item Let $J$ be a proper subset of $\BS_0$. Then $\supp\left({}^{\pr(K)}W_0 \right) \subseteq J$ if and only if $C_{K}^+ \subseteq W_J \left(\overline{C^+}\right):= \bigcup_{x\in W_J}x\left(\overline{C^+}\right)$ (this follows from (\ref{eq:1.3})).
    \item if $\supp\left({}^{\pr(K)}W_0 \right) = \BS_0$, then $\supp\left({}^{\pr(K')}W_0 \right) = \BS_0$ for any $K' \subsetneq K$. 
\end{enumerate}

For exceptional types, we use the computer program to compute the set $\supp\left({}^{\pr(K)}W_0 \right)$ by brutal force. For classical types, we first give a specific description of the chamber $C_{K}^+$ case by case, then consider certain ``sufficiently non-dominant" elements in $C_K^+$. If these elements cannot be translated into $\overline{C^+}$ by elements with proper supports in $W_0$, then by $(a)$, we have $\supp\left({}^{\pr(K)}W_0 \right) =\BS_0$. Combined with $(b)$, we can rule out most of the cases where $\supp\left({}^{\pr(K)}W_0 \right) =\BS_0$ (see \S\ref{sec:1.7.1}--\S\ref{sec:1.7.4}). 

We use Bourbaki's notation. We write $\ud a$ (resp. $\ud a'$) for the sequence $(a_1,a_2,\ldots,a_n)$ (resp. $(a_1',a_2',\ldots,a_n')$).

\subsubsection{Type $A_{n-1}$ case ($n\ge2$)}\label{sec:1.7.1}
We have $X_*(T) = \bigoplus_{j=1}^{n} \BZ \e_j$, $\a_i = \e_i-\e_{i+1}$ for $i=1,2\ldots,n-1$ and $\a_0 = -\th = -\e_1+\e_n$. Note that all indices $i$ are special. Recall that we only consider non-special $K$.

Let $K = \tilde\BS - \{i_1,i_2 \}$ with $0\le i_1 < i_2 \le n-1$. If $i_1 = 0$, then ${}^{\pr(K)}(W_0) = {}^{ \BS_0-\{i_2\} }(W_0)$ and it is easy to see that $\supp\left({}^{\BS_0 - \{i_2\} }W_0 \right) = \BS_0$. Assume $i_1>0$. Then
\begin{align*}
    C_{K} = \bigl\{ \ud a\in \BZ^n \mid & a_1 > \cdots > a_{i_1}, a_{i_1+1} > \cdots > a_{i_2}, a_{i_2+1} > \cdots >a_n > a_{1}\}\\
    C_{K}^+ = \bigl\{ \ud a\in \BZ^n \mid & a_1 > \cdots > a_{i_1}, a_{i_1+1} > \cdots > a_{i_2}, a_1 > a_{i_2+1} > \cdots >a_n  \}.
\end{align*}
Take $\ud a  \in C_{K}^+$ with $ a_{i_2} > a_{1}$. It is easy to see that any $x\in W_0$ with $x(\ud a) \in \overline{C^+}$ must have support $\supseteq \{1,2,\ldots,i_2-1\}$. Take $\ud a'  \in C_{K}^+$ such that $ a'_{n} > a'_{i_1+1}$. It is easy to see that any $x\in W_0$ such that $x(\ud a') \in \overline{C^+}$ must have support $\supseteq \{ i_1+1,\ldots,n-1\}$. By $(a)$, we have $\supp\left({}^{\pr(\tilde\BS - \{i\})}W_0 \right) = \BS_0$ (here, the elements $\ud a$ and $\ud a'$ are ``sufficiently non-dominant").

By $(b)$, we conclude that there is no choice of $K$ such that $(W_0,K)$ satisfies condition (*).

\subsubsection{Type $B_n$ case ($n\ge2$)}\label{sec:1.7.2}
We have $X_*(T) = \bigoplus_{j=1}^{n} \BZ \e_j$, $\a_i = \e_i-\e_{i+1}$ for $i=1,2\ldots,n-1$, $\a_n = \e_n$ and $\a_0 = -\th = -\e_1-\e_2$. Note that the index $i$ is special if and only $i = 0,1$. 

Let $i\in\{2, 3, \ldots,n-1\}$. We have
\begin{align*}
    C_{\tilde\BS - \{i\}} = \bigl\{ \ud a\in \BZ^n \mid & a_1 > a_2> \cdots > a_i, a_1+a_2<0, a_{i+1} > a_{i+2}  > \cdots > a_n > 0 \bigr\};\\
    C_{\tilde\BS - \{i\}}^+ = \bigl\{ \ud a\in \BZ^n \mid & a_1 > a_2> \cdots > a_i, a_{i-1}+a_{i}>0, a_{i+1} > a_{i+2}  > \cdots > a_n > 0 \bigr\}.
\end{align*}
Take $\ud a  \in C_{\tilde\BS - \{i\}}^+$ such that $ a_i<0$ and $a_1<a_n$ (this kind of element is ``sufficiently non-dominant"). It is not difficult to see that there is no $x\in W_0$ with proper support such that $x(\ud a) \in \overline{C^+}$. By $(a)$, we have $\supp\left({}^{\pr(\tilde\BS - \{i\})}W_0 \right) = \BS_0$.

Let $i = n$. It is easy to see that for any $\ud a\in C_{\tilde\BS- \{i\}}^+$, there exists some $x\in W$ with support $\subseteq \{n\}$ such that $x(\ud a)\in \overline{C^+}$ (indeed, we only need to deal with the case where $a_n<0$). By $(a)$, $\supp\left({}^{\pr(\tilde\BS- \{i\})}W_0 \right) = \{n\}$ is proper. 

For $K = \tilde\BS-\{0,1\}$, $\tilde\BS-\{0,n\}$ or $\tilde\BS-\{1,n\}$, we can similarly prove $\supp\left({}^{\pr(\tilde\BS- \{i\})}W_0 \right) =\BS_0$. 

By $(b)$, we conclude that $(W_0,K)$ satisfies condition (*) if and only if $K = \tilde\BS - \{n\}$.

\subsubsection{Type $C_n$ case ($n\ge2$)}\label{sec:1.7.3}
We have $X_*(T) = \bigoplus_{j=1}^{n} \BZ \e_j$, $\a_i = \e_i-\e_{i+1}$ for $i=1,2\ldots,n-1$, $\a_n = 2\e_n$ and $\a_0 = -\th = -2\e_1$. Note that the index $i$ is special if and only if $i = 0,n$. 

Let $i\in\{1, 2, 3, \ldots,n-1\}$. We have
\begin{align*}
     C_{\tilde\BS - \{i\}} = \bigl\{ \ud a\in \BZ^n \mid&  0>a_1 > a_2 > \cdots > a_i, a_{i+1} > a_{i+2}> \cdots > a_n > 0 \bigr\};\\
     C_{\tilde\BS - \{i\}}^+ = \bigl\{ \ud a\in \BZ^n \mid&  a_1 > a_2  > \cdots > a_i > 0, a_{i+1} > a_{i+2}> \cdots > a_n > 0\bigr\}.
\end{align*}
It is easy to see that for any $\ud a\in C_{\tilde\BS - \{i\}}^+$, there exists some $x\in W$ with support $\subseteq  \{1,2,\ldots,n-1\}$ such that $x(\ud a)\in \overline{C^+}$. By $(a)$, $\supp\left({}^{\pr(\tilde\BS - \{i\})}W_0 \right) = \{1,2,\ldots,n-1\} $ is proper. 

Let $K = \tilde\BS - \{i_1,i_2 \}$ with $0\le i_1 < i_2 \le n$. We have
\begin{align*}
     C_{K}^+ = \bigl\{ \ud a\in \BZ^n \mid&  a_1 > a_2 > \cdots > a_{i_1}>0, a_{i_1+1} > \cdots > a_{i_2}, a_{i_2+1}> \cdots > a_n > 0 \bigr\}.
\end{align*}
Take $\ud a  \in C_{K}^+$ such that $ a_{i_1+1}<-a_1$ (this kind of elements is ``sufficiently non-dominant"). It is not hard to see that there is no $x\in W_0$ with proper support such that $x(\ud a) \in \overline{C^+}$. By $(a)$, we have $\supp\left({}^{\pr(K)}W_0 \right) = \BS_0$.

By $(b)$, we conclude that $(W_0,K)$ satisfies condition (*) if and only if $K = \tilde\BS - \{i\}$ for $i = 1,2,\ldots,n-1$.


\subsubsection{Type $D_n$ case ($n\ge4$)}\label{sec:1.7.4}
We have $X_*(T) = \bigoplus_{j=1}^{n} \BZ \e_j$, $\a_i = \e_i-\e_{i+1}$ for $i=1,2\ldots,n-1$, $\a_n = \e_{n-1} + \e_n$ and $\a_0 = -\th = -\e_1 -\e_2$. Note that the index $i$ is special if and only if $i = 0,1,n-1,n$.

Let $i\in\{  2, 3, \ldots,n-2\}$. We have
\begin{align*}
    C_{\tilde{\BS}-\{i\}} = \bigl\{ \ud a\in \BZ^n \mid& a_{1} > a_2>  \cdots > a_i  , a_1+a_2<0,\\
                                        & a_{i+1} > a_{i+2}>  \cdots > a_n,a_{n-1}+a_{n} >0  \bigr\};\\
    C_{\tilde{\BS}-\{i\}}^+ = \bigl\{ \ud a\in \BZ^n \mid& a_{1} > a_2> \cdots > a_i , a_{i-1}+a_i>0,\\
                                            &a_{i+1} > a_{i+2}>  \cdots > a_n,a_{n-1}+a_{n} >0 \bigr\}.
\end{align*}
Take $\ud a  \in C_{\tilde{\BS}-\{i\}}^+$ such that $ a_1<a_{n-1}$ and $a_i<a_n<0$ (this kind of element is ``sufficiently non-dominant"). It is not difficult to see that there is no $x\in W_0$ with proper support such that $x(\ud a) \in \overline{C^+}$. By $(a)$, we have $\supp\left({}^{\pr(\tilde{\BS}-\{i\})}W_0 \right) = \BS_0$. 

Let $K = \tilde\BS-\{n-1,n\}$. Take $\ud a  \in C_{K}^+$ such that $  a_{n } > a_1$. Then any $x\in W_0$ such that $x(\ud a) \in \overline{C^+}$ must have support $\supseteq \BS_0-\{n\}$. Take $\ud a  \in C_{K}^+$ such that $  a_{n } < - a_1$. Then any $x\in W_0$ such that $x(\ud a) \in \overline{C^+}$ must have support $\supseteq \BS_0-\{n-1\}$. By $(a)$, we have $\supp\left({}^{\pr(K)}W_0 \right) = \BS_0$. 

For $K = \tilde\BS - \{0,1\}$, $\tilde\BS - \{0,n-1\}$, $\tilde\BS - \{0,n\}$, $\tilde\BS - \{1,n-1\}$ or $\tilde\BS - \{1,n\}$, we can similarly prove that $\supp\left({}^{\pr(K)}W_0 \right) =\BS_0$. 

By $(b)$, we conclude that there is no choice of $K$ such that $(W_0,K)$ satisfies condition (*).

\section{Changing parahoric levels}\label{sec:2}
\subsection{Main result}\label{sec:2.1}
Let $K_1\subseteq K_2$ be two spherical subsets of $\tilde\BS$. Let $\breve \CK_1$ and $\breve \CK_2$ be the corresponding standard parahoric subgroup of $\breve G$. The natural projection $\breve G/\breve \CK_1\longrightarrow \breve G / \breve \CK_2$ induces 
$$\pi = \pi_{K_1,K_2}: \CA_{K_1}(\bG,\mu) \longrightarrow \CA_{K_2}(\bG,\mu).$$ 
The main result of this section is the following.
\begin{theorem}\label{thm:main}
For any $m\in \CA_{K_2}(\bG,\mu)$, the fiber $\pi^{-1}(m)$ isomorphic to a single Schubert variety in the partial flag variety $\breve \CK_2/\breve \CK_1$. In particular, each fiber is irreducible.      
\end{theorem}

The strategy for proving Theorem \ref{thm:main} is as follows. In \S\ref{sec:2.2}, we first interpret the fiber of $\pi$ as a union of Schubert cells and then reduce it to Proposition \ref{prop:max}, which is a purely combinatorial statement. In \S\ref{sec:2.3}, we study the combinatorial property of the admissible set under a geometric point of view using the notions of alcoves, galleries, and acute cones, and finally prove Proposition \ref{prop:max} in \S\ref{sec:2.4}.

\subsection{Fiber as union of Schubert cells}\label{sec:2.2}
We keep the notation in Theorem \ref{thm:main}. Note that
\begin{align*}\tag{2.1}\label{eq:2.1}
\CA_{K_2}(\bG,\mu) =& \bigcup_{w\in\Adm(\mu)_{K_2}}\breve \CI \dot{w} \breve \CI/ \breve \CK_2 = \bigsqcup_{w \in {}^{K_1}\tW^{K_2}\cap \Adm(\mu)_{K_2}}\breve \CK_1 \dot{w} \breve \CK_2 / \breve \CK_2\\
                                        =& \bigsqcup_{w \in {}^{K_1}\tW^{K_2}\cap \Adm(\mu)} \breve \CK_1 \dot{w} \breve \CK_2 / \breve \CK_2.
\end{align*}
Here, the last equality follows from Theorem \ref{thm:AdmmuK}.

Let $m\in\CA_{K_2}(\bG,\mu)$. By (\ref{eq:2.1}), we may write $m$ as $k_1 \dot{w} \breve \CK_2/\breve \CK_2$ with $k_1\in\breve \CK_1$ and $w\in {}^{K_1}\tW^{K_2}\cap \Adm(\mu)$. By definition of $\pi$, we have
\begin{align*}
\pi^{-1}(k_1 \dot{w} \breve \CK_2/\breve \CK_2) & = \left( k_1 \dot{w} \breve \CK_2/\breve \CK_1\right)   \cap \CA_{K_1}(\bG,\mu)  \\
                            &\cong \bigl\{  k_2\in \breve \CK_2 \bigm| k_1 \dot{w} k_2 \in \bigcup_{w'\in\Adm(\mu)}\breve \CK_1 w' \breve \CK_1 \bigr\} /\breve \CK_1  \\
                            &= \bigcup_{x\in W_{K_2} } \bigl\{k_2\in \breve \CI x \breve \CI \bigm| k_1 \dot{w} k_2 \in \bigcup_{w'\in\Adm(\mu)_{K_1}}\breve \CI w' \breve \CI \bigr\} /\breve \CK_1  \\
                            & = \bigcup_{ \substack{  x\in W_{K_2}\\ wx \in \Adm(\mu)_{K_1} } }\breve \CI x \breve \CI /\breve \CK_1.
\end{align*}
Here, the last equality follows from the assumption that $w\in {}^{K_1}\tW^{K_2}$. By Theorem \ref{thm:AdmmuK}, we have $\Adm(\mu)_{K_1} = \Adm(\mu)W_{K_1}$. Hence, for any $x\in W_{K_2}$, $wx\in \Adm(\mu)_{K_1}$ if and only if $wxx_1\in \Adm(\mu)$ for some $x_1\in W_{K_1}$. Thus,
\begin{align*}\tag{2.2}\label{eq:2.2}
\pi^{-1}(k_1 \dot{w} \breve \CK_2/\breve \CK_2) & \cong \bigcup_{ \substack{  x\in W_{K_2}\\ wx \in \Adm(\mu) } } \bigcup_{x_1\in W_{K_1}} \breve \CI x x_1 \breve \CI /\breve \CK_1 =  \bigcup_{ \substack{  x\in W_{K_2}\\ wx \in \Adm(\mu) } }\breve \CI x \breve \CK_1 /\breve \CK_1.
\end{align*}
In particular, $\pi^{-1}(k_1 \dot{w} \breve \CK_2/\breve \CK_2)$ is isomorphic to a union of Schubert varieties in the partial flag $\breve \CK_2/\breve \CK_1$. Provided (\ref{eq:2.2}), to prove Theorem \ref{thm:main}, it suffices to prove the following purely combinatorial statement.
\begin{proposition}\label{prop:max}
Let $K$ be a spherical subset of $\tilde \BS$. Let $w\in \Adm(\mu) $. Then the set $w W_K \cap \Adm(\mu)$ contains a unique maximal element.
\end{proposition}

\begin{remark}
    After the paper was finished, we showed it to Felix Schremmer. He pointed out that this proposition can also be proved using \cite[Theorem 4.2]{Sch24} and \cite[Lemma 2.13]{Sch23}. 
\end{remark}

\subsection{Acute cones}\label{sec:2.3}
The admissible set $\Adm(\mu)$ is a complicated combinatorial object. It is convenient to think of it from a geometric point of view, using the notions of alcoves, galleries, and acute cones, which we now briefly review. 

Recall that in \S\ref{sec:1.1}, by choosing a special vertex of the base alcove $\mathfrak{a}$, we identify the apartment $\CA$ with $V = X_*(T)_{\G_0}\otimes_{\BZ}\BR$. By definition, alcoves are connected components of $V - \bigcup_{\tilde{\a}}H_{\tilde{\a}}$, where $\tilde{\a}$ runs over the set of affine roots $\Phi_{\af} \cong \Sigma\times \BZ$ and $H_{\tilde{\a}}$ is the root hyperplane in $V$ corresponding to $\tilde{\a}$. Note that $\tW$ acts on the set of alcoves naturally. Although the action is not transitive in general, it is natural to view elements in $\tW$ as alcoves in $\CA$.

Let $\mathfrak{a}'$ and $\mathfrak{a}''$ be two alcoves. A gallery from $\mathfrak{a}'$ to $\mathfrak{a}''$ with length $r$ is a sequence of alcoves $\mathfrak{b}_0 = \mathfrak{a}', \mathfrak{b}_1, \mathfrak{b}_2, \ldots, \mathfrak{b}_r = \mathfrak{a}''$ such that $\mathfrak{b}_i$ and $\mathfrak{b}_{i+1}$ share a common wall for $i = 0,1,\ldots,r-1$. The gallery is said to be minimal if there is no gallery from $\mathfrak{a}'$ to $\mathfrak{a}''$ with a shorter length. Let $w$ be an element of the affine Weyl group $W_{\af}$. Then there is a natural bijection between the set of reduced expressions of $w$ and the set of minimal galleries from $\mathfrak{a}$ to $w(\mathfrak{a})$. That is, any reduced expression $w = s_1s_2\cdots s_r $ can be associated with the minimal gallery $\mathfrak{a}, s_1(\mathfrak{a}), s_{1}s_2(\mathfrak{a}),\ldots, s_1s_2\cdots s_r(\mathfrak{a})$.

We now introduce the notion of acute cones.
\begin{definition}[{\cite[\S2.5]{HH17}}, {\cite[\S5]{HN02}}]
Let $z\in W_0$.
\begin{enumerate}
    \item For any root hyperplane $H$ in $V$, define $H^{z+}$ as the connected component of $V - H$ that contains sufficiently deep alcoves in the Weyl chamber $z(C^+)$ (recall that $C^+$ is the dominant chamber). More explicitly, suppose that $H$ corresponds to the affine root $(\a,k)$ with $\a\in z(\Phi^+)$ (otherwise replace $\a$ with $-\a$), then $H^{z+} = \{v\in V \mid \<v,\a\> > -k\}$. 
    
    
    \item We say a gallery $\mathfrak{b}_0, \mathfrak{b}_1, \mathfrak{b}_2, \ldots, \mathfrak{b}_r$ is in the $z$-direction, if $\mathfrak{b}_i \in (H_i)^{z+}$ for $i = 1,2,\ldots,r$, where $H_i$ is the common wall of $\mathfrak{b}_{i-1}$ and $\mathfrak{b}_i$.
    
    \item Let $\CC(\mathfrak{a}, z)$ be the set of alcoves $\mathfrak{b}$ such that there exists a gallery from the base alcove $\mathfrak{a}$ to $\mathfrak{b}$ in the $z$-direction. We call $\CC(\mathfrak{a}, z)$ the acute cone in the $z$-direction.
\end{enumerate}
\end{definition}
By \cite[Corollary 5.6]{HN02}, any alcove lies in some (but may not be unique) acute cone. The following lemma provides a more specific description of the acute cones.

\begin{lemma}[{\cite[Lemma 5.3]{HN02}}]\label{lem:gallery}
Let $\mathfrak{a}'$ and $\mathfrak{a}''$ be two alcoves and $z\in W_0$. Then the following are equivalent:
\begin{enumerate}
    \item There exists a gallery from $\mathfrak{a}'$ to $\mathfrak{a}''$ in the $z$-direction;
    \item There exists a minimal gallery from $\mathfrak{a}'$ to $\mathfrak{a}''$ in the $z$-direction;
    \item All minimal galleries from $\mathfrak{a}'$ to $\mathfrak{a}''$ are in the $z$-direction;
    \item For all root hyperplane $H$ separating $\mathfrak{a}'$ and $\mathfrak{a}''$, we have $\mathfrak{a}'' \in H^{z+}$.
\end{enumerate}
\end{lemma}

From a geometric point of view, roughly speaking, the set $\Adm(\mu)$ is equal to the convex hull of $W_0(\mu)$. More precisely, $\Adm(\mu)$ corresponds to the set of alcoves at the intersection of the ``opposite obtuse cones" $B(t^{z(\ud\mu)}(\mathfrak{a}),z)$ for all $z\in W_0$ (see \cite{HN02} and \cite{HH17}).

One of the main difficulties in studying $\Adm(\mu)$ is that, for a given $w\in\Adm(\mu)$, it is hard to determine for which $z\in W_0$, we have $w\le t^{z(\ud\mu)}$. The following result by Haines and the first-named author provides a criterion for this using acute cones.
\begin{theorem}[{\cite[Corollary 4.4]{HH17}}]\label{thm:HH}
Let $w\in \Adm(\mu)$ and $z\in W_0$. Assume that $w(\mathfrak{a})\in \CC(\mathfrak{a},z)$. Then $w\le t^{z(\ud\mu)}$.
\end{theorem}

\subsection{Proof of Proposition \ref{prop:max}}\label{sec:2.4}


Let $K$ be a spherical subset of $\tilde \BS$ and $w\in \Adm(\mu)$. We shall prove that $wW_K\cap\Adm(\mu)$ contains a unique maximal element. The key to the proof is to show that

(a) there exists $z\in W_0$ such that the acute cone $\CC(\mathfrak{a},z)$ contains the whole left coset $wW_K$, i.e., $w'(\mathfrak{a})\in \CC(\mathfrak{a},z)$ for all $w'\in wW_K$.

Let $w''$ be the maximal element in the left coset $wW_K$. By \cite[Corollary 5.6]{HN02}, there exists $z\in W_0$ such that $w''(\mathfrak{a})\in \CC(\mathfrak{a},z)$. Let $w'\in wW_K$. As $w''$ is maximal in $wW_K$, there exist $i_1, i_2, \ldots, i_r \in K$ such that $w'' = w' s_{i_1}s_{i_2}\cdots s_{i_r}$ and $\ell(w'') = \ell(w') + r$. Hence, there exists a minimal gallery from $\mathfrak{a}$ to $w''(\mathfrak{a})$ which passes through $w'(\mathfrak{a})$. Since $w''\in\CC(\mathfrak{a},z)$, by Lemma \ref{lem:gallery}, the gallery is in the $z$-direction. Then we obtain a (minimal) gallery from $\mathfrak{a}$ to $w'(\mathfrak{a})$ which is also in the $z$-direction. Hence $w'(\mathfrak{a})\in \CC(\mathfrak{a},z)$ by definition. This proves (a).

By Theorem \ref{thm:HH} and (a), for any $w' \in wW_K \cap \Adm(\mu) $, we have $w' \le t^{z(\ud \mu)}$. In other words, we have $wW_K\cap \Adm(\mu) = \{w'\in wW_K\mid w'\le t^{z(\ud \mu)}\}$. By \cite[Theorem 2.2]{Deod87}, the set $ \{w'\in wW_K\mid w'\le t^{z(\ud \mu)}\}$ contains a unique maximal element, and so does $wW_K\cap \Adm(\mu)$. This proves Proposition \ref{prop:max}.

\section{The hyperspecial case}\label{sec:3}
A natural question is how to describe the element $\max(w W_{K} \cap  \Adm(\mu) )$ in Proposition \ref{prop:max}. The question is difficult in general. In this section, we consider the case where $K =   \BS_0$.

We point out the following left-right symmetry. For any $w\in \Adm(\mu)$, set $w^*=w^{-1}$ and $\mu^* = -w_0(\mu)$ where $w_0$ is the longest element in $W_0$. Then the inverse automorphism sends $ \Adm(\mu)$ to $  \Adm(\mu^*)$ and gives a bijection
\begin{align*}\tag{3.1}\label{eq:3.1}
w W_0 \cap \Adm(\mu) \tilde{\longrightarrow}  W_0 w^* \cap \Adm(\mu^*).
\end{align*}
As we will see later, the notations become simpler once we consider the right coset version. In this section, we shall study the right coset version instead of the original left coset version. 

The main result of this section is Theorem \ref{thm:algorithm} which provides an efficient and explicit algorithm to compute the element $\max\{x\in W_0 \mid xw \in \Adm(\mu)\}$ for $w\in {}^{\BS_0}\tW $. The major tools we use are the quantum Bruhat graph and the Damazure product. 

\subsection{Quantum Bruhat graph}\label{sec:3.1}

The quantum Bruhat graph $\mathrm{QBG}(\Sigma)$ is a directed graph, whose set of vertices is the Weyl group $W_0$ and whose edges are of the form $x\rightarrow xs_\a$ for $x\in W_0$ and $\a\in \Sigma^+$ whenever one of the following conditions is satisfied:
\begin{enumerate}

\item $\ell(xs_\a) = \ell(x)+1$ or
\item $\ell(xs_\a) = \ell(x)+1-\langle \a^\vee,2\rho\rangle$, where $\a^{\vee}$ is the coroot corresponding to $\a$.
\end{enumerate}
The edges satisfying the first condition are called \emph{Bruhat} edges and the edges satisfying the second condition are called \emph{quantum} edges. The weight of each Bruhat edge is defined to be zero. The weight of a quantum edge $x\rightarrow xs_\a$ is defined as the coroot $\a^{\vee}$. The weight of a path is defined as the sum of the weights of all edges of the path. It is easy to see that $\text{QBG}(\Sigma)$ is path-connected. 


We recall some basic properties of the quantum Bruhat graph.
\begin{lemma}[{\cite[Lemma 1]{Pos05}}]\label{lem:Pos}
Let $x,x'\in W_0$.
\begin{enumerate}
    \item All shortest paths from $x$ to $x'$ in $\mathrm{QBG}(\Sigma)$ have the same weight. We denote this weight by $\wt(x, x')$.
    \item Any path from $x$ to $x'$ has weight $\ge$ $\wt(x,x')$.
    \item $\wt(x ,x') = 0$ if and only if $x \le x'$.
\end{enumerate}
\end{lemma}
By definition of the weight function, we have 
\begin{align*}\tag{3.2}\label{eq:3.2}
\wt(x,x'') \le \wt(x,x')+ \wt(x',x'') \text{ for any } x,x',x''\in W.
\end{align*}

The following is the quantum Bruhat graph of type $A_2$ (the Bruhat edges are drawn as solid lines and the quantum edges are drawn as dashed lines):

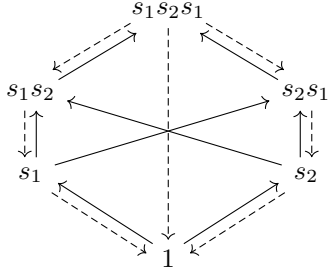
\def\qecolor{dashed}
\def\seshift{0.5ex}
\def\quantumEdge{dashed}
\def\shortEdgeShiftRight{0.5ex}
\begin{align*}
\begin{tikzcd}[ampersand replacement=\&,column sep=2em]
\&s_1 s_2 s_1\ar[ddd,\qecolor]
\ar[dl,\qecolor,shift right=\seshift]\ar[dr,\qecolor,shift left=\seshift]\\
s_1 s_2\ar[ur,shift right=\seshift]\ar[d,\qecolor,shift right=\seshift]\&\&
s_2 s_1\ar[ul,shift left=\seshift]\ar[d,\qecolor,shift left=\seshift]\\
s_1\ar[u,shift right=\seshift]\ar[urr]\ar[dr,\qecolor,shift right=\seshift]\&\&s_2\ar[u,shift left=\seshift]\ar[ull]\ar[dl,\qecolor,shift left=\seshift]\\
\&1\ar[ru,shift left=\seshift]\ar[lu,shift right=\seshift]
\end{tikzcd}
\end{align*}

We have the following description of the admissible set, which was first established in \cite{HY} and later generalized in \cite{Sch24}. The $\l = \mu$ case is first proved in \cite{HL15}.
\begin{proposition}[{\cite[Proposition 3.3]{HY}} and {\cite[Proposition 4.12]{Sch24}}]\label{prop:qbg}
Let $\l\in X_*(T)^+$ and $I(\l) = \{i\in\BS_0 \mid \<\ud\l,\a_i\>=0\}$. Then 
$$\Adm(\mu)\cap W_0 t^{\ud\l }W_0 = \{x  t^{\ud\l } y  \mid x \in W_0 , y  \in {}^{I(\l )}W_0, \wt(x , y^{-1}) \le \ud\mu -\ud\l  \}.$$
In particular, we have $$\Adm(\mu)\cap W_0 t^{\ud\mu }W_0 = \{x  t^{\ud\mu } y  \mid x \in W_0 , y  \in {}^{I(\mu )}W_0, x\le y^{-1}\}.$$
\end{proposition}
Note that elements in ${}^{\BS_0}\tW $ are of the form $w = t^{\ud\l }y$ with $\l \in X_*(T)^+$ and $y\in {}^{I(\l)}W_0$ (see for example \cite[\S9.1]{He14}). It follows that 
\begin{align*}\tag{3.3}\label{eq:3.3}
\max\{x\in W_0\mid x t^{\ud\l} y \in\Adm(\mu) \} = \max \{ x \in W_0 \mid \wt(x,y^{-1}) \le \ud\mu - \ud\l\}.
\end{align*}

In the case of $w = t^{\ud\mu}y$ with $y\in {}^{I(\mu)}W_0$, using Lemma \ref{lem:Pos} (3), the equation (\ref{eq:3.3}) becomes 
$$\max\{x\in W_0\mid x w \in\Adm(\mu) \} = \max \{ x \in W_0 \mid x \le y^{-1} \}.$$ 
Combining this with (\ref{eq:2.2}) and the left-right symmetry (\ref{eq:3.1}), we obtain the following.

\begin{corollary}\label{cor:sch}
Assume $\bG$ splits over $L$ and $\ud\mu$ is regular, then any Schubert variety in the flag variety $\breve\CK_0/\breve\CI$ appears as a fiber of the level changing map from the Iwahori level to the hyperspecial level. Here $\breve\CK_0$ is the standard parahoric subgroup of $\breve G$ corresponding to $\BS_0$.
\end{corollary}

By Corollary \ref{cor:sch}, the fibers of the level changing map from the Iwahori level to the hyperspecial level are not always smooth. On the other hand, if $\bG=GL_n$, then it is easy to see that all the fibers of the level changing map from level $K_{0,1} = \tilde{\BS}- \{0, 1\}$ to the hyperspecial level $K_0 =  \BS_0 $ are smooth. It is an interesting question to determine for which level changing maps, all the fibers are smooth. 

\smallskip

Combining Proposition \ref{prop:qbg} with Proposition \ref{prop:max}, we obtain the following property of the quantum Bruhat graph, which is of independent interest.
\begin{corollary}\label{cor:qbg}
Let $\g\in \sum_{i\in \BS_0} \BZ_{\ge0} \a_i^{\vee}$ and $v\in W_0$. Then the set $$\{ x \in W_0 \mid \wt(x ,v) \le \g\}$$ contains a unique maximal element.
\end{corollary}
\begin{proof}
Let $\ud\l\in X_*(T)_{\G_0}$ be sufficiently dominant regular such that $\ud\l +\g$ is dominant. Let $\mu $ be a lift of $\ud\l + \g$ in $X_*(T)^+$. Let $w =  t^{\ud\mu} v^{-1} \in {}^{\BS}\tW$. By Proposition \ref{prop:qbg}, we have $\{ x \in W_0 \mid \wt(x ,v) \le \g\} = \{ x \in W_0 \mid xw\in \Adm(\mu) \}$. By Proposition \ref{prop:max}, and the left-right symmetry (\ref{eq:3.2}), this set contains a unique maximal element.
\end{proof}


The weight function $\wt(x,v)$ for $x,v\in W_0$ is in general very difficult to compute. There is a recursive algorithm to reduce it to the $v=1$ case as follows. Let $s_i $ be a simple reflection in $W_0$ such that $s_iv<v$. By \cite[Corollary 3.3]{Sad23} or \cite[Lemma 7.7]{LNSSS}, we have 
\begin{align*}\tag{3.4}\label{eq:3.4}
\wt(x,v) = \wt(\min\{x,s_i x\},s_iv).    
\end{align*}
Let $\g\in \sum_{i\in \BS_0} \BZ_{\ge0} \a_i^{\vee}$. By (\ref{eq:3.4}), we have 
\begin{align*}\tag{3.5}\label{eq:3.5}
\max\{x \in W_0 \mid \wt(x,v)\le \g\} &= \max\{x\in W_0\mid \wt(\min\{x,s_i x\},s_iv) \le \g \}\\ &= \max\{x',s_ix'\},  
\end{align*}
where $x' = \max\{x \in W_0\mid \wt(x,s_iv)\le \g\}$.

For computing the weight function $\wt(x,1)$, we have the following result.
\begin{proposition}[{\cite[Proposition 4.11]{MV20}}]\label{prop:downward}
For any $x\in W_0$, there exists a shortest path from $x$ to $1$ which is purely downward, i.e. consists only of quantum edges.
\end{proposition}

\subsection{Demazure product}\label{sec:3.2}
To state Theorem \ref{thm:algorithm}, we need an additional tool, namely, the Demazure product.

Let $u,v\in W_0$. By \cite[Lemma 1]{He09}, each of the three sets $\{u'v'\mid u'\le u,v'\le v\}$, $\{u'v\mid u'\le u \}$ and $\{uv'\mid  v'\le v\}$ contains a unique maximal element, and their maxima are the same. We denote the common maximum by $u*v$ and call it the Demazure product of $u$ and $v$. Note that $(W,*)$ is a (associative) monoid which is uniquely determined by the following two conditions
\begin{enumerate}
\item $u * v = u v$ for any $u,v\in W_0$ such that $\ell(uv) = \ell(u) + \ell(v)$;
\item $s_i * v = v$ for any $i\in \BS_0$ and $v\in W_0$ such that $s_i v<v$.
\end{enumerate}

The following Lemma involving both weight function of quantum Bruhat graph and Demazure product serves as a key step in the proof of Theorem \ref{thm:algorithm}. The proof we give here is due to Felix Schremmer, which is simpler than our original proof. 
\begin{lemma}\label{lem:qbg}
Let $z\in W_0$ and $\b\in \Sigma^+$. Then $\wt(z*s_{\b},z) \le \b^{\vee}$. 
\end{lemma}
\begin{proof}
By (\ref{eq:3.2}), we have 
\begin{align*}
\wt(z*s_{\b},z) \le \wt (z*s_{\b}, (z*s_{\b})s_\b ) + \wt((z*s_{\b})s_\b,z ).   
\end{align*}
By \cite[Lemma 4.3]{Sch23}, the first term is $\le \b^{\vee}$. Set $z' = (z*s_{\b})s_\b$. Then $z's_\b = z*s_{\b}$. By the definition of Demazure product, we have $z'\le z$. Hence $\wt((z*s_{\b})s_\b,z )=0$ by Lemma \ref{lem:Pos} (3).
\end{proof}

\subsection{Algorithm}\label{sec:3.3} By (\ref{eq:3.3}) and (\ref{eq:3.5}), computing the element $\max\{x\in W_0\mid xw\in\Adm(\mu)\}$ for $w\in {}^{\BS_0}\tW$ can be reduced to computing the element $\max\{x\in W_0 \mid \wt(x,1)\le \g\}$ for $\g \in   \sum_{i\in \BS_0} \BZ_{\ge0} \a_i^{\vee}$.

Let $\g\in \sum_{i\in \BS_0} \BZ_{\ge0} \a_i^{\vee}$. We construct an element $z_\g\in W_0$ inductively as follows. If $\g$ is a simple coroot, define $z_{\g} = s_{\g}$, the reflection corresponding to $\g$. Assume that $\g$ is not a simple coroot and that $z_{\g'}$ has been defined for all $\g'<\g$. Let $\b$ be a maximal element in $\Sigma^+$ such that $\b^{\vee} \le \g$. Define $z_{\g} =  z_{\g-\b^{\vee}} * s_{\b}  $.

We point out that the element $z_{\g}$ is well defined, i.e., independent of the choice of the maximal root $\b$. Following \cite[\S4.1]{BM15}, a sequence $(\b_1,\b_2,\ldots, \b_m)$ of roots in $\Sigma^+$ is said to be a \emph{greedy decomposition} of $\g$, if $\b_i$ is a maximal element in $\Sigma^+$ such that $\b_i^{\vee}\le \g - \b_1^{\vee} - \cdots - \b_{i-1}^{\vee}$ for $i = 1,2,\ldots,m$. In this case, it is proved in loc. cit. that $s_{\b_i}*s_{\b_j} = s_{\b_j}*s_{\b_i}$ for all $1\le i,j\le m$. Moreover, any two greedy decompositions of $\g$ are equal up to reordering. Hence, the element $z_{\g}$ is well-defined. This element also coincides the element defined in \cite[Definition 4.6]{BM15}. \footnote{In \cite[Definition 4.6]{BM15}, the element $z_{\g}$ is defined as $s_{\b}*z_{\g-\b^{\vee}}$ inductively. The reason that we define $z_{\g} =  z_{\g-\b^{\vee}} * s_{\b}$ is to make it more convenient for the proof of Theorem \ref{thm:algorithm}, as we use right multiplication in the definition of quantum Bruhat graph. Note that the Demazure product $\ast$ is denoted by $\cdot$ in \cite{BM15}.}


\begin{theorem}\label{thm:algorithm}
Let $\g\in \sum_{i\in \BS_0} \BZ_{\ge0} \a_i^{\vee}$. Then $$z_{\g}=\max\{ x \in W_0 \mid \wt(x,1) \le \g\}.$$
\end{theorem}

\begin{remark}
Felix Schremmer helps us finds this algorithm in the Sagemath source code and asks Mark Shimozono about a reference. Mark Shimozono then provides us the reference \cite{BM15}. Since the proof of Theorem \ref{thm:algorithm} is not written down in any literature, we give a proof here.
\end{remark}

\begin{proof}[Proof of Theorem \ref{thm:algorithm}]
Assume that $\g$ is a simple coroot. By Proposition \ref{prop:downward}, it is easy to see that $ \{ x \in W_0 \mid \wt(x,1) \le \g\} = \{1,s_{\g}\}$. Note that $z_{\g} = s_{\g}$ by definition. The statement is clear. Assume that $\g$ is not a simple coroot and that $z_{\g'} = \max\{x\in W_0 \mid \wt(x,1)\le \g'\}$ for any $\g'<\g$. Let $\b$ be a maximal element in $\Sigma^+$ such that $\b^{\vee} \le \g$. We have
\begin{align*}
 \wt(z_{\g } , 1) &=\wt(z_{\g-\b^{\vee}}*s_{\b} , 1)\\
                  &\le \wt(z_{\g-\b^{\vee}}*s_{\b}, z_{\g-\b^{\vee}}) + \wt(z_{\g-\b^{\vee}},    1) \\
                  &\le \b^{\vee} + \g-\b^{\vee}\\
                  & = \g.
\end{align*}
Here, the first inequality follows from (\ref{eq:3.2}), and the second inequality follows from Lemma \ref{lem:qbg} and the induction hypothesis.

Denote $x_0 = \max\{x\in W_0 \mid \wt(x,1) \le \g\}$. It suffices to prove that $x_0 \le z_{\g }$. By Proposition \ref{prop:downward}, there exists a shortest path $x_0 \rightarrow x_1 \rightarrow \cdots \rightarrow 1$ such that all edges are quantum edges. Let $\a^{\vee}$ be the weight of the first edge of the path. Then $x_1 = x_0 s_{\a}$ and $\ell(x_0) = \ell(x_1) + \ell(s_\a)$ by the definition of quantum edge. Hence $x_0 = x_1 * s_{\a}$. Note also that $\wt(x_1,1) = \g - \a^{\vee}$. Hence $x_1 \le z_{\g - \a^{\vee}}$ by the induction hypothesis. Note that $\wt(s_\a,1) \le \a^{\vee}$ by \cite[Lemma 4.3]{Sch23}. Therefore, $s_{\a}\le z_{\a^{\vee}}$ by the induction hypothesis. Then we have
\begin{align*}
x_0 = x_1 * s_{\a} \le z_{\g-\a^{\vee}} * z_{\a^{\vee}} \le z_{\g},    
\end{align*}
where the last inequality follows from \cite[Theorem 4.11 (b)]{BM15}. This completes the proof.
\end{proof}
\begin{remark}
The algorithm in Theorem \ref{thm:algorithm} does not involve computing the weight function of the quantum Bruhat graph. It involves only computing the maximal roots $\beta$ and certain Demazure products, which is much easier than handling the Quantum Bruhat graph.
\end{remark}
    
We give some examples of how the algorithm in Theorem \ref{thm:algorithm} works. We use Bourbaki's notation.
\begin{example}
In the case of type $A_2$, let $\g = \a_1^{\vee} + 2\a_2^{\vee}$. Take the maximal root $\b = \a_1+\a_2$. Then $z_{\g} = z_{\g - \b^{\vee}} * s_{\b} = 
s_2 *( s_1s_2s_1) = s_1s_2s_1$.
\end{example}
\begin{example}
In the case of type $C_2$, let $\g = \a_1^{\vee} + 2\a_2^{\vee}$. Take the maximal root $\b = 2\a_1+\a_2$. Note that $\b^{\vee} = \a_1^{\vee}+\a_2^{\vee}$ and $s_{\b} = s_1s_2s_1$. Then $z_{\g} = z_{\g - \b^{\vee}} * s_{\b} = z_{\a_2^{\vee}} * (s_1 s_2 s_1) = s_2 s_1 s_2 s_1$.  
\end{example}
In fact, the element $z_{\g}$ in Theorem \ref{thm:algorithm} is an involution, i.e., $(z_{\g})^{-1}=z_{\g}$. This is proved \cite[Corollary 4.9]{BM15}.

\printbibliography
\end{document}